\documentclass[10pt,reqno]{article}
\usepackage{amsfonts}
\usepackage{amsmath}
\usepackage{amssymb}
\usepackage{epsfig}
\usepackage{amsthm}
\usepackage[english]{babel}
\usepackage[hypertex]{hyperref}
\newcommand{\N}{\mathbb N}

\newcommand{\R}{\mathbb R}

\newcommand{\E}{\mathbb E}
\newcommand{\pr}{\mathbb P}

\newcommand{\hier}[3][X]{#1^{(#2)}_{#3}}

\newcommand{\half}{\frac{1}{2}}
\newcommand{\salj}{\mathcal{F}}
\newcommand{\cals}{\mathcal{S}}

\begin{document}
\theoremstyle{plain}
\newtheorem{thm}{Theorem}
\newtheorem{lem}[thm]{Lemma}
\newtheorem{cor}[thm]{Corollary}
\newtheorem{prop}[thm]{Proposition}
\title{Geometric preferential attachment in non-uniform metric spaces}

\author{Jonathan Jordan\\
 University of Sheffield}
\maketitle
\begin{abstract}We investigate the degree sequences of geometric preferential attachment graphs in general compact metric spaces.  We show that, under certain conditions on the attractiveness function, the behaviour of the degree sequence is similar to that of the preferential attachment with multiplicative fitness models investigated by Borgs et al.  When the metric space is finite, the degree distribution at each point of the space converges to a degree distribution which is an asymptotic power law whose index depends on the chosen point.  For infinite metric spaces, we can show that for vertices in a Borel subset of $S$ of positive measure the degree distribution converges to a distribution whose tail is close to that of a power law whose index again depends on the set.  \\ AMS 2010 Subject Classification: Primary 05C82, Secondary 60D05. \\ Key words and phrases:geometric random graphs; preferential
attachment.\end{abstract}

\section{Introduction}

In \cite{FFV1} and \cite{FFV2}, Flaxman, Frieze and Vera
introduced a model for a growing graph driven by geometric
preferential attachment. In this model, which is a variant of the
Barab\'{a}si-Albert preferential attachment model introduced in
\cite{scalefreedefs} and analysed in \cite{sfdiameter, brst},
vertices are given a random location in a metric space $S$ and the probability that a new vertex is connected
to an already existing vertex $u$ depends on the distance between
them in space as well as on the degree of $u$.  The aim is to combine the ideas of the Barab\'{a}si-Albert model with some of those found in spatial graph models, for example random geometric graphs, \cite{penrosergg}, online models such as the FKP model, \cite{bbbcr}, and its special case the online nearest neighbour graph, \cite{wadeong}, which is also a degenerate case of geometric preferential attachment.

In \cite{geopref4}, under a rather strong condition on the probability measure determining the locations of the vertices
and the strength of the effect of distance on the connection
probabilities, it was shown that the limiting proportion of vertices with degree $d$ was the same as that found for the Barab\'{a}si-Albert model in \cite{brst}.

In this paper we consider one of the questions which was left open in \cite{geopref4}, namely what it is possible to say if we weaken the assumptions on the  probability measure $\mu$ determining the locations of the vertices.  In \cite{geopref4} it was required that for any fixed $r$, $\mu(B_r(x))$ is constant as a function of $x$, where $B_r(x)$ is the open ball (in the underlying metric on $S$) of radius $r$
centred on $x$, for example the case of Haar measure on a compact group with an invariant metric.  We will show that if this assumption is weakened then (still assuming certain conditions on the the strength of the effect of distance on the connection
probabilities and still assuming that the space is compact) the behaviour of the degree distribution of the model is similar to that found for preferential attachment with multiplicative fitness, as investigated by Borgs et al in \cite{fitness}.  In that paper each vertex has a random fitness, and the probability of a new vertex connecting to an existing vertex is proportional to the product of its degree and its fitness.

In addition we generalise the model of \cite{geopref4} so that the attractiveness of a vertex at location $x$ to one at location $y$ is not necessarily the same as the attractiveness of a vertex at location $y$ to one at location $x$.  We do this by replacing the attractiveness $F(\rho(x,y))$ (where $\rho$ is the metric on $S$) by a function of two variables $\alpha(x,y)$.  This also allows the preferential attachment with fitness model of \cite{fitness} to be seen as a special case of our model, by allowing $\alpha(x,y)$ to only depend on $x$.

\section{Our model and results}\label{model}

We assume $S$ is a compact
metric space with metric $\rho$ and probability measure $\mu$; the
locations of the added vertices will be assumed to be independent
random variables $(X_n)_{n\in\N}$ with law $\mu$.  We denote the Borel $\sigma$-algebra of $S$ by $\mathcal{B}(S)$.

Let $\alpha:S\times S\to\R^+$ be an attractiveness function; we will usually assume
that $\alpha$ is continuous.  The interpretation of $\alpha(x,y)$ here is that it is the attractiveness of a vertex at $x$ to a new vertex at $y$.  Note that the situation in \cite{geopref4}, where the attractiveness was defined as $F(\rho(x,y))$ with $F$ being a function from $\R^+\to\R^+$, can be treated using the formulation in this paper by letting $\alpha(x,y)=F(\rho(x,y))$.  Also note that the preferential attachment with fitness model of \cite{fitness} can be seen as a special case of this model with $S$ being the set of possible fitnesses, the ``location'' of a vertex being its fitness, and $\alpha(x,y)=x$ for all $x,y\in S$.

Let $m\in\N$ be the number of vertices that each new vertex
will be connected to when it is added to the graph, which as in other papers, such as \cite{scalefreedefs, sfdiameter, geopref4}, will be a parameter of the model.

In \cite{geopref4} it was assumed that the metric space $S$ and measure $\mu$ satisfied the condition that, for any fixed $r$, $\mu(B_r(x))$ is constant as a function
of $x$, where $B_r(x)$ is the open $\rho$-ball of radius $r$
centred on $x$, a rather strict ``uniformity'' condition.  The aim here is to discuss what happens when this assumption does not hold.

To start the process, we let $G_0$ be a connected graph with $n_0$
vertices and $e_0$ edges, and we give each vertex $v\in V(G_0)$ a location $X_v\in S$. Then, to form $G_{n+1}$ from $G_n$, we add a new vertex $v_{n+1}$ to the graph whose location $X_{n+1}=X_{v_{n+1}}$ is a random variable on $S$ with law $\mu$ independent of $X_1,X_2,\ldots,X_n$ and the structure of $G_n$.  Let
$\hier[V]{n+1}{\iota}$, $1\leq \iota\leq m$ be the random variables
representing the $m$ vertices chosen to be neighbours of the new
vertex $v_{n+1}$ at time $n+1$. Conditional on $X_{n+1}$ and $\salj_n$,
where $\salj_n$ is the $\sigma$-algebra generated by the graphs
$G_0,G_1,\ldots,G_n$ and the location in space of their vertices,
we let the $\hier[V]{n+1}{\iota}$ be chosen independently such that for $v\in V(G_n)$
the probability that $\hier[V]{n+1}{\iota}=v$ is
$$\frac{\deg_{G_n}(v)\alpha(X_v,X_{n+1})}{D_n(X_{n+1})},$$ where
$\deg_{G}(v)$ is the degree of the vertex $v$ in the graph $G$ and
$$D_n(x)=\sum_{u\in V(G_n)}\deg_{G_n}(u)\alpha(X_u,x).$$  Note that we
allow that $\hier[V]{n+1}{\iota_1}=\hier[V]{n+1}{\iota_2}$ for some $\iota_1 \neq
\iota_2$, in which case multiple edges will form, but we do not allow loops.

Following \cite{geopref4}, we define a sequence of (random) measures $\delta_n$ on $S$ by, for $A\in \mathcal{B}(S)$, $$\delta_n(A)=\frac{1}{2(mn+e_0)}\sum_{v\in A\cap
V(G_n)}\deg_{G_n}(v),$$ so that $\delta_n(A)$ is the total degree in $A$, normalised to give a probability measure on $S$. Then
$$\frac{D_n(x)}{2(mn+e_0)}=\int_S \alpha(x,y)\;d\delta_n(x).$$  Thus the probability that $\hier[V]{n+1}{i}=v$ given $G_n$ and $X_{n+1}=x$ can be rewritten as
$$\frac{\deg_{G_n}(v)}{2(mn+e_0)}\frac{\alpha(X_v,x)}{\int_S \alpha(x,y)\;d\delta_n(x)},$$ and so the probability that $\hier[V]{n+1}{i}=v$ given $G_n$ is \begin{equation}\label{choosek}\frac{\deg_{G_n}(v)}{2(mn+e_0)}\int_S\frac{ \alpha(X_v,x)}{\int_S \alpha(x,y)\;d\delta_n(x)}\;d\mu(x).\end{equation}

Given the uniformity and symmetry assumptions in \cite{geopref4} $\delta_n$ converges weakly to $\mu$, and furthermore $\int_S \alpha(x,y)\;d\mu(x)$ does not depend on $x$.  This can be used to show that the integral in \eqref{choosek} tends to $1$, so for large $n$ the probability of choosing $v_k$ is close to what it would be in standard preferential attachment.

Our aim is to show that if the uniformity assumption does not hold, then under certain conditions $\delta_n$ converges weakly to some measure $\nu$ on $S$.  If this holds then for large $n$ then $\pr(\hier[V]{n+1}{i}=v)$ will be approximately $$\frac{\deg_{G_n}(v)}{2(mn+e_0)}\int_S\frac{ \alpha(X_v,y)}{\int_S \alpha(x,y)\;d\nu(x)}\;d\mu(y).$$  We have the following theorem.

\begin{thm}\label{overall} Assume that $\alpha$ is a continuous function from $S \times S$ to $\R^{+}$, that $\log \alpha$ is Lipschitz in both components, and that there exists $\alpha_0>0$ such that $\alpha(x,y)\geq \alpha_0$ for all $x,y \in S$.  Then we have the following.\begin{enumerate}
\item There exists a probability measure $\nu$ on $S$ such that $\delta_n$ converges weakly to $\nu$.
\item Define, for $u\in S$, $$\phi(u)=\int_S\frac{ \alpha(u,y)}{\int_S \alpha(x,y)\;d\nu(x)}\;d\mu(y);$$ then $$\pr(\hier[V]{n+1}{i}=v)\frac{2(mn+e_0)}{\deg_{G_n}(v)}=\phi(X_v)+o(1),$$\end{enumerate}\end{thm}
Note that, unlike in \cite{geopref4}, we do not allow $\alpha(x,y)\to\infty$ as $\rho(x,y)\to 0$ here.  In the case where $S$ is a finite metric space, we do not require the conditions on $\alpha$; in that case we can simply assume that $\alpha$ is a function from $\R^{+}\cup \{0\}$ to $\R^{+}\cup \{0\}$.

That $\pr(\hier[V]{n+1}{i}=v)$ is approximately proportional to a constant $\phi(X_v)$ times the degree of $v$ is reminiscent of the preferential attachment with fitness model of Borgs et al in \cite{fitness}.  This suggests that in our model the quantity $\phi(X_v)$, which depends on the location of the vertex $v$, might play a similar role to the fitness in the model of \cite{fitness}.  Indeed, in \cite{fitness} it is shown that vertices of different fitnesses have degree distributions following approximately power law distributions where the index of the power law depends on the fitness, and we will show similar results here, starting with the case where $S$ is a finite metric space.

\begin{thm}\label{finitedd}
Assume that the metric space $S$ consists of a finite number of points $z_1,z_2,\ldots,z_N$, and that $\alpha(x,y)$ is a function from $S\times S$ to $\R^{+}\cup \{0\}$.  Let $\mu_i=\mu(\{z_i\})$, and let $\phi_i=\phi(z_i)$ where $\phi$ is as defined in Theorem \ref{overall}.  Then if we let $\hier[p]{n}{d,i}$ be the proportion of vertices of $G_n$ which are located at $z_i$ and have degree $d$, we have
$$\hier[p]{n}{d,i} \to \frac{2\mu_i}{\phi_i}\frac{\Gamma(m+2\phi_i^{-1})\Gamma(d)}{
\Gamma(m)\Gamma(d+2\phi_i^{-1}+1)},$$ almost surely, as $n\to\infty$.
\end{thm}

In a similar way to in \cite{fitness} this gives an asymptotic (in $d$) power law for vertices at the same location.  As in \cite{fitness} the tail index of the power law depends on $\phi_i$, and here is $2\phi_i^{-1}$.   Indeed we can think of $\phi_i$ as giving a ``geometrical fitness'' for a point at $z_i$, thought of as a multiplicative fitness in the sense of \cite{fitness}.

We can also obtain a result on the degree sequence in the case where $S$ is infinite.

\begin{thm}\label{infinitedd}
Let $A\in\mathcal{B}(S)$ be a Borel set with $\mu(A)>0$.  Let $\hier[p]{n}{d,A}$ be the proportion of vertices in $G_n$ which are of degree $d$ and have locations in $A$.  Let $\phi_A=\sup_{x\in A}\phi(x)$ and $\psi_A=\inf_{x\in A}\phi(x)$.

Then, almost surely, $$\liminf_{n\to\infty} \frac{1}{\mu(A)}\sum_{k=m}^{d} \hier[p]{n}{k,A}\geq \frac{2}{\phi}\frac{\Gamma(m+2\phi^{-1})}{\Gamma(m)}\sum_{k=m}^{d} \frac{\Gamma(k)}{
\Gamma(k+2\phi^{-1}+1)},$$ for any $\phi>\phi_A$, and
$$\limsup_{n\to\infty} \frac{1}{\mu(A)}\sum_{k=m}^{d} \hier[p]{n}{k,A}\leq \frac{2}{\phi}\frac{\Gamma(m+2\phi^{-1})}{\Gamma(m)}\sum_{k=m}^{d} \frac{\Gamma(k)}{
\Gamma(k+2\phi^{-1}+1)},$$ for any $\phi<\psi_A$.
\end{thm}

This shows that for vertices in $A$ the limiting tail behaviour of the degree sequence is between those of power laws with tail indices $2\psi_A^{-1}$ and $2\phi_A^{-1}$.

To prove the above theorems, we will start off by considering, in section \ref{finite}, the case where $S$ is a finite metric space, where stochastic approximation techniques can be used to show the convergence of the measures (which in the finite case are points in a simplex).  In section \ref{infinite}, we will then use a coupling between geometric preferential attachment on $S$ and a process closely related to the finite space case to show that convergence of measures also applies in the infinite case, and deduce Theorem \ref{infinitedd}.  The arguments are to some extent based on those in \cite{fitness} but use more general stochastic approximation techniques.  Before that, we will give some examples where the measure $\nu$ can be found explicitly.

\subsection{Examples}

\subsubsection{Uniform measure}

In \cite{geopref4}, it was assumed that for any fixed $r$, $\mu(B_r(x))$ is constant as a function of $x$, where $B_r(x)$ is the open ball (in the underlying metric on $S$) of radius $r$ centred on $x$.  This includes for example the case of Haar measure on a compact group with an invariant metric.  It was also assumed that $\alpha(x,y)=F(\rho(x,y))$ for some function $F$.  Under these assumptions, the results in \cite{geopref4} imply that the measure $\nu$ in Theorem \ref{overall} is equal to $\mu$.  As a result the function $\phi$ defined in the statement of Theorem \ref{overall} is $1$ everywhere on $S$, and all subsets of $S$ with positive measure under $\mu$ have the same limiting degree distribution for their vertices.

\subsubsection{Preferential attachment with fitness}

As mentioned above, the preferential attachment with fitness model of \cite{fitness} can be considered as a special case of our model by letting the set $S$ be a subset of $\R^+$, the location of a vertex being equal to its fitness, and taking $\alpha(x,y)=x$, so that the attractiveness of a vertex does not depend on the location of the new vertex but simply on its own location, that is its fitness.  (Note that for Theorem \ref{overall} to apply to this model as stated, we require that the set of fitnesses be bounded away from zero and to be contained within a compact subset of $\R^+$.)

This model is analysed in detail in \cite{fitness}, considering finite, discrete countable and continuous fitness distributions separately.  For example in Theorems 6 and 7 of \cite{fitness} it is assumed that the fitness distribution is defined by a probability density function $g(x)$ on an interval $[0,h]$, and that $g(x)$ is non-zero on $(0,h)$.  It is also assumed that $m=1$.  Under these assumptions they show that, if $M_{n,[a,b]}$ the number of edge endpoints in $G_n$ with fitnesses in $[a,b]$ then for $0\leq a<b<h$ $$\frac{M_{n,[a,b]}}{n}\to \nu_{[a,b]}$$ almost surely as $n\to\infty$, where an explicit formula for $\nu_{[a,b]}$ is given.  For example if there is a solution $\lambda_0\geq h$ to \begin{equation}\label{fgr}\int_0^h \frac{x g(x)}{\lambda_0-x}\;dx\geq 1\end{equation} (this is described in \cite{fitness} as the ``fit get richer'' phase) then $\nu_{[a,b]}$ is defined as $\lambda_0\int_a^b \frac{g(x)}{\lambda_0-x}$.  In our notation, this shows that the measure $\nu$ satisfies $\nu([a,b])=\frac{\nu_{[a,b]}}{2}$.

In what is described in \cite{fitness} as the ``innovation pays off'' phase, where there is no solution to \eqref{fgr}, the results given in Theorem 7 of \cite{fitness} show that the measure $\nu$ has an atom at $h$.  More generally, it is possible that in our setting there may be subsets of $S$ for which the measure $\nu$ is positive but $\mu$ is not, indicating that a proportion of vertices tending to zero have a positive proportion of the edge ends, the innovation pays off phase of \cite{fitness} being the simplest example where this happens.

\subsubsection{A two-point metric space}

We let $S=\{0,1\}$, $\alpha(0,0)=\alpha(1,1)=1$  and $\alpha(0,1)=\alpha(1,0)=a>0$.  We define the measure $\mu$ by $\mu(\{0\})=p$ and $\mu(\{1\})=1-p$, with $0<p<1$.  This is then a special case of the framework in section \ref{finite}, with $N=2$.  In this setting the simplex can be represented as $[0,1]$, and the Lyapunov function which we use in section \ref{finite} can be written as $$V(y)=1-\half\left[p(\log(y)+\log(y+a(1-y)))-(1-p)(\log(1-y)+\log(1-y+ay)\right].$$  Let $y_0$ be the unique root in $(0,1)$ of \begin{equation}\label{N2fp}p\left(\frac{1}{y}+\frac{1-a}{y+a(1-y)}\right)=(1-p)\left(\frac{1}{1-y}+\frac{1-a}{1-y+ay}\right),\end{equation} which gives the location of the minimum of $V(y)$ in $[0,1]$; then Theorem \ref{overall} applies with the measure $\nu$ being defined by $\nu(\{0\})=y_0$ and $\nu(\{1\})=1-y_0$.  We can also calculate $$\phi(0)=\frac{p}{y_0+(1-y_0)a}+\frac{(1-p)a}{1-y_0+y_0a}$$ and $$\phi(1)=\frac{1-p}{1-y_0+y_0a}+\frac{pa}{y_0+(1-y_0)a}.$$

We note that if $y=p$ then \eqref{N2fp} only holds if either $a=1$ (which would be equivalent to standard preferential attachment) or $p=1/2$ (n which case the uniformity assumption of \cite{geopref4} would hold) so if neither of these two conditions hold the measures $\nu$ and $\mu$ are different.

\section{The finite case}\label{finite}

This section will prove Theorem \ref{overall} in the finite case and Theorem \ref{finitedd}.  We assume the metric space $S$ consists of a finite number of points $z_1,z_2,\ldots,z_N$.  Let $\mu_i=\mu(\{z_i\})$, and let $a_{i,j}=\alpha(z_i,z_j)$.

Let $\hier[Y]{n}{i}$ be the number of edge ends at point $z_i$ in $G_n$, i.e. the sum $\sum_{v\in V(G_n),X_v=z_i}\deg_n(v)$.  Then let $\hier[y]{n}{i}=\delta_n(\{z_i\})=\frac{\hier[Y]{n}{i}}{2(mn+e_0)}$, i.e. the proportion of edge ends which are located at $z_i$, and let $\hier[y]{n}{}$ be the point $(\hier[y]{n}{1},\hier[y]{n}{2},\ldots,\hier[y]{n}{N})$ in the $N$-simplex.

Then we can write $$\pr(\hier[V]{n+1}{\iota}=v|X_{n+1}=j) = \frac{\deg_{G_n}(v)a_{X_v,j}}{\sum_{k=1}^N  a_{k,j}\hier[Y]{n}{k}}$$ and so the probability that $\hier[V]{n+1}{\iota}$ is at $z_i$, conditional on the new vertex being at $z_j$, is $$\frac{\hier[Y]{n}{i}a_{i,j}}{\sum_{k=1}^N \hier[Y]{n}{k}a_{k,j}}=\frac{\hier[y]{n}{i}a_{i,j}}{ \sum_{k=1}^N \hier[y]{n}{k} a_{k,j}}.$$

Then $$\E(\hier[Y]{n+1}{i}|\salj_n)=\hier[Y]{n}{i}+ m\mu_i+ m\sum_{j=1}^N \mu_j \frac{\hier[y]{n}{i}a_{i,j}}{ \sum_{k=1}^N \hier[y]{n}{k} a_{k,j}},$$ and so $$\E(\hier[y]{n+1}{i}|\salj_n)=\hier[y]{n}{i}\frac{2(mn+e_0)}{2(m(n+1)+e_0)}+ \frac{m}{2(m(n+1)+e_0)}\mu_i+ \frac{m}{2(m(n+1)+e_0)}\sum_{j=1}^N \mu_j \frac{\hier[y]{n}{i}a_{i,j}}{\sum_{k=1}^N \hier[y]{n}{k}a_{k,j}},$$ giving \begin{eqnarray*} \E(\hier[y]{n+1}{i}|\salj_n)-\hier[y]{n}{i} &=& \frac{2}{2(n+1+e_0/m)}\left(\half\mu_i+\half \sum_{j=1}^N \mu_j \frac{\hier[y]{n}{i}a_{i,j}}{\sum_{k=1}^N \hier[y]{n}{k}a_{k,j}}-\hier[y]{n}{i}\right) \\ &=& \frac{2}{2(n+1+e_0/m)}\left(g_i(\hier[y]{n}{})-\hier[y]{n}{i}\right),\end{eqnarray*} where $g$ is a map from the $N$-simplex to itself given by the $i$ co-ordinate being $$g_i(y)=\half\mu_i+\half \sum_{j=1}^N \mu_j \frac{a_{i,j} y_i}{\sum_{k=1}^N y_k a_{k,j}}.$$  Alternatively $$ \E(\hier[y]{n+1}{i}|\salj_n)-\hier[y]{n}{i} = \frac{2}{2(n+1+e_0/m)}G_i(\hier[y]{n}{i})$$ where $G(y)=g(y)-y$ and so its components are given by $$G_i(y)=\half\mu_i+\half \sum_{j=1}^N \mu_j \frac{a_{i,j} y_i}{\sum_{k=1}^N y_k a_{k,j}}-y_i.$$

\begin{prop}\label{lyapunov}There exists $\nu$ in the $N$-simplex such that as $n\to\infty$ we have $\hier[y]{n}{}\to \nu$, almost surely.
\end{prop}

\begin{proof}
For $y$ in the interior of the $N$-simplex, let $$V(y)=1-\half\sum_{j=1}^N\mu_j\left(\log y_j+\log\sum_{k=1}^N y_ka_{k,j}\right).$$
Then (using $1=\sum_{j=1}^N y_j$), $$G_i(y)=-y_i\frac{\partial}{\partial y_i}V(y),$$ and as $y_i>0$ this means that $V$ is a Lyapunov function for $G$. By concavity of the logarithm $V$ is a convex function and it tends to infinity near the boundary of the $N$-simplex, so it has a unique minimum, at a point which we will call $\nu$, in the interior of the $N$-simplex.

Proposition 2.18 of \cite{pemantlesurvey} now gives $\hier[y]{n}{}\to \nu$ a.s. as $n\to\infty$.
\end{proof}
Proposition \ref{lyapunov} corresponds to Proposition 2 in Section 3 of \cite{fitness}.

\begin{prop}\label{phifinite}If vertex $v$ is at location $z_i$ then $$\frac{2(mn+e_0)}{\deg_{G_n}(v)}\pr(\hier[V]{n+1}{\iota}=v)=\phi_i+o(1),$$ where $$\phi_i := \sum_{j=1}^N \mu_j \frac{a_{i,j}}{\sum_{k=1}^N a_{k,j}\nu_k}=2-\frac{\mu_i}{\nu_i}.$$
\end{prop}
\begin{proof} By Proposition \ref{lyapunov}, $$\frac{2(mn+e_0)}{\deg_{G_n}(v)}\pr(\hier[V]{n+1}{\iota}=v|X_{n+1}=j)=\frac{a_{i,j}}{\sum_{k=1}^N a_{k,j}\nu_k}+o(1),$$ giving $$\frac{2(mn+e_0)}{\deg_{G_n}(v)}\pr(\hier[V]{n+1}{\iota}=v)=\sum_{j=1}^N \mu_j \frac{a_{i,j}}{\sum_{k=1}^N a_{k,j}\nu_k}+o(1)=\phi_i+o(1).$$

To show that $\phi_i=2-\frac{\mu_i}{\nu_i}$, we note that the probability that $\hier[V]{n+1}{\iota}$ is a vertex at location $z_i$ is then $\nu_i\phi_i +o(1)$, which implies \begin{equation}\label{sumnuphi}\sum_{i=1}^N \nu_i \phi_i=1.\end{equation} and, because each edge has one end at a new vertex (location chosen with law $\mu$) and one end at a vertex $\hier[V]{n}{\iota}$ for some $n$ and $\iota$, also implies
\begin{equation}\label{nui}\nu_i=\half(\mu_i+\nu_i\phi_i)\end{equation} and from \eqref{nui} we get $\phi_i=2-\frac{\mu_i}{\nu_i}$.\end{proof}

Propositions \ref{finite} and \ref{lyapunov} prove Theorem \ref{overall} in the finite case.

Let $\hier[p]{n}{d,i}$ be the proportion of the vertices of $G_n$ which are located at $z_i$ and have degree $d$.  We will use the above to show an asymptotic power law for $\hier[p]{n}{d,i}$.  We will need to use the following lemma based on Lemma 2.6 and Corollary 2.7 of \cite{pemantlesurvey}.

\begin{lem}\label{convlemma}
For $n\in \N$, let $A_n$ and $B_n$ be random variables taking non-negative values, $\xi_n$, $R_n$ random variables taking real values, and $k$ a positive constant, such that
$$B_{n+1}-B_{n}=\frac{1}{n}(A_n-k B_n+\xi_n)+R_{n+1}$$ and \begin{enumerate}
\item $A_n\to a$ as $n\to\infty$, almost surely; \item $\sum_{n=1}^{\infty}R_n<\infty$; \item $\E(\xi_n)=0$ and $\xi_n$ is bounded.\end{enumerate} Then $B_n\to \frac{a}{k}$ as $n\to
\infty$, almost surely.
\end{lem}

\begin{proof}
Fix $\delta>0$ and $\epsilon>0$.  We note that if $|A_n-a|<\epsilon$ and $B_n<\frac{a}{k}-\frac{\delta+\epsilon}{k}$ then $A_n-k B_n>\delta$.  The proof of Lemma 2.6 of \cite{pemantlesurvey} and the fact that $|A_n-a|<\epsilon$ if $n$ is large enough now shows that $(B_n)_{n\in\N}$ cannot visit $[0,\frac{a}{k}-2\frac{\delta+\epsilon}{k}]$ infinitely often.  Similarly if $|A_n-a|<\epsilon$ and $B_n>\frac{a}{k}+\frac{\delta+\epsilon}{k}$ then $A_n-k B_n<-\delta$ and so $(B_n)_{n\in\N}$ cannot visit $[\frac{a}{k}+2\frac{\delta+\epsilon}{k},\infty)$ infinitely often.  Hence, $B_n \to \frac{a}{k}$, almost surely.
\end{proof}

We can now prove Theorem \ref{finitedd}.

\textit{Proof of Theorem \ref{finitedd}.} We start by
showing that, when $m>1$, the probability of a multiple edge being
formed at time $n$ from a vertex of degree $d$ is $O(n^{-1})$.  Conditional on $\salj_n$ and
$X_{n+1}=j$, the probability that a vertex $u$ is connected to the
new vertex at least twice is bounded above by
$$
\binom{m}{2}\left(\frac{\deg_{G_n}(u)a_{X_u,j}}{\sum_{k=1}^N  a_{k,j}\hier[Y]{n}{k}}\right)^2
$$ so, conditional on $\salj_n$, it is bounded above by
$$\binom{m}{2}\left(\frac{\deg_{G_n}(u)}{2(mn+e_0)}\right)^2\E\left(
\left(\frac{a_{X_u,X_{n+1}}}{\sum_{k=1}^N
a_{k,X_{n+1}}\hier[y]nk} \right)^2|\salj_n\right),$$ and proposition \ref{lyapunov} ensures that the expectation here is bounded.

Using Proposition \ref{phifinite}, for each $d$,
\begin{equation}\label{proportions}\E(\hier[p]{n+1}{d,i}|\salj_n)=
\frac{1}{n+n_0+1}\left(\hier[p]{n}{d,i}
\left(n+n_0-\phi_i\frac{d}2+o(1)\right)+\hier[p]{n}{d-1,i}
\left(\phi_i\frac{d-1}{2}+o(1)\right)+\mu_i\delta_{m,d}\right)+O(n^{-2})\end{equation}

(where $\delta_{m,d}=1$ if $d=m$ and is zero otherwise) and so
\begin{equation}\label{proportions2}\E(\hier[p]{n+1}{d,i}|\salj_n)-\hier[p]{n}{d,i}=\frac{1}{n+n_0+1}\left(
-\hier[p]{n}{d,i}\left(1+\phi_i\frac{d}2+o(1)\right)+\hier[p]{n}{d-1,i}
\left(\phi_i\frac{d-1}{2}+o(1)\right)+\mu_i\delta_{m,d}\right)+O(n^{-2}).\end{equation}  If $d=m$ \eqref{proportions2} becomes $$\E(\hier[p]{n+1}{m,i}|\salj_n)-\hier[p]{n}{m,i}=\frac{1}{n+n_0+1}\left(
-\hier[p]{n}{m,i}\left(1+\phi_i\frac{m}2+o(1)\right)+\mu_i\right)+O(n^{-2}).$$
so Lemma \ref{convlemma}, with $B_n=\hier[p]{n}{m,i}$, $A_n=\mu_i+o(1)$ and $k=1+\phi_i\frac{m}{2}$, gives
$$\hier[p]{n}{m,i}\to\frac{2\mu_i}{2+m\phi_i},$$ almost surely, as
$n\to\infty$.  When $d>m$ \eqref{proportions2} becomes $$\E(\hier[p]{n+1}{d,i}|\salj_n)-\hier[p]{n}{d,i}=\frac{1}{n+n_0+1}\left(
-\hier[p]{n}{d,i}\left(1+\phi_i\frac{d}2+o(1)\right)+\hier[p]{n}{d-1,i}
\left(\phi_i\frac{d-1}{2}+o(1)\right)\right)+O(n^{-2}),$$ and repeatedly using Lemma \ref{convlemma}, with $B_n=\hier[p]{n}{d,i}$, $A_n=\hier[p]{n}{d-1,i}\left(\phi_i\frac{d-1}{2}+o(1)\right)$ and $k=1+\phi_i\frac{d}{2}$, gives
$$\hier[p]{n}{d,i}\to\frac{2\mu_i}{2+m\phi_i}\prod_{j=m+1}^{d}\frac{j-1}{2\phi_i^{-1}+j}=\frac{2\mu_i}{\phi_i}\frac{\Gamma(m+2\phi^{-1})\Gamma(d)}{
\Gamma(m)\Gamma(d+2\phi^{-1}+1)},$$ almost surely, as
$n\to\infty$.

\qed

\section{The infinite case}\label{infinite}

\subsection{Coupling}\label{coupling}

We partition $S$ into a finite set $\cals=\{S_i, i=1,2,\ldots,N_{\cals}\}$ with each $S_i\in \mathcal{B}(S)$.  Later on we will have each $S_i$ having diameter at most some small value.

We now construct a modified graph process, which will be similar to a finite space geometric preferential attachment, as in section \ref{finite}, on $\{0,1,2,\ldots,N_{\cals}\}$.  The extra point $0$ will be used to allow the construction of a coupling, similar to that in Appendix C of \cite{fitness}, with the geometric preferential attachment on $S$.

For $1 \leq i,j \leq N_{\cals}$, define \begin{eqnarray*}a_{i,j} &=& \sup_{u\in S_i,w\in S_j}\alpha(u,w) \\ b_{i,j} &=& \inf_{u\in S_i,w\in S_j}\alpha(u,w) \\ \gamma_{i,j} &=& b_{i,j}/a_{i,j} \\ \gamma_{\cals} &=& \inf_{1 \leq i,j \leq N_{\cals}} \gamma_{i,j}.\end{eqnarray*}  Also define $\mu_i = \mu(S_i)$ for $i=1,2,\ldots,N_{\cals}$, $\mu_{0}=0$, $h = \sup_{1 \leq i,j \leq N_{\cals}} a_{i,j}$, $a_{i,0}=a_{0,i}=h$ for all $i$.

We will construct a sequence of graphs $(\hier[G]{\cals}n)_{n\in\N}$ whose vertices have locations in $\{0,1,2,\ldots,N_{\cals}\}$ as follows.  We start with a graph $\hier[G]{\cals}0$, and construct $\hier[G]{\cals}{n+1}$ from $\hier[G]{\cals}n$.  Let the total degree of vertices at location $i$ after $n$ steps be $\hier[Y]{n,\cals}{i}$.  At each step we first add a new vertex, which is at location $\hier[X]{\cals}{n+1}$, where $\hier[X]{\cals}{n+1}$ is a random variable taking the value $j$ with probability $\mu_j$.  Conditional on $\hier[X]{\cals}{n+1}=j$, we then add $m$ edges from the new vertex which connect to $m$ vertices $\hier[V]{n+1,\cals}{\iota}$, $\ell=1,2,\ldots,m$ which are chosen independently of each other, with for a vertex $v$ at location $i$ $$\pr(\hier[V]{n+1,\cals}{\iota}=v)=\frac{\deg_{\hier[G]{\cals}n}(v)a_{i,j}\gamma_{\cals}}{\sum_{k=0}^{N_{\cals}} \hier[Y]{n,\cals}{k}a_{k,j}}$$ for $0\leq i\leq N_{\cals}$.  For each edge this leaves a probability  $(1-\gamma_{\cals})$ that it does not connect to any existing vertex.  If this happens, a new vertex is created at location $0$, and the edge connects there.  The interpretation here is that each of the $m$ edges the new vertex tries to connect an old vertex with probability proportional to its degree times an attractiveness factor based on the locations, but that the connection is only accepted with probability $\gamma_{\cals}$, and if the connection is rejected then a new vertex is formed for the connection.  Note that the extra vertices added then behave as the other vertices, with attractiveness $h$ to all other locations.  Following section \ref{finite} let $\hier[y]{n,\cals}{i}=\frac{\hier[Y]{n,\cals}{i}}{2(mn+e_0)}$, the proportion of the total degree at location $i$.

We now show that the geometric preferential attachment process $(G_n)_{n\in\N}$ on $S$ can be coupled to the above process.  In the geometric preferential attachment process, let $\hier[Y]{n}{i}$ be the total degree of the vertices in $S_i$, and let $\hier[y]{n}{i}=\frac{\hier[Y]{n}{i}}{2(mn+e_0)}=\delta_{n}(S_i)$.  We aim to couple the two processes so that we always have $\hier[Y]{n,\cals}{i}\leq \hier[Y]{n}{i}$ for all $1\leq i\leq N_{\cals}$.  To start with, let $\hier[G]{\cals}0$ and $G_0$ be the same graph, with the location $\hier[X]{\cals}{v}$ of a vertex $v$ in $\hier[G]{\cals}0$ being the $i$ such that $X_v\in S_i$ where $X_v$ is the location of the corresponding vertex in $G_0$.  This ensures $\hier[Y]{0,\cals}{i}= \hier[Y]{0}{i}$ for $1\leq i\leq N_{\cals}$.  Then, we claim that the coupling can be done so that $\hier[Y]{n,\cals}{i}\leq \hier[Y]{n}{i}$ implies that $\hier[Y]{n+1,\cals}{i}\leq \hier[Y]{n+1}{i}$ for $1\leq i\leq N_{\cals}$.  Given that $X_{n+1}\in S_j$, that is that the new vertex in $G_{n+1}$ is in $S_j$, which occurs with probability $\mu_j$, let $\hier[X]{\cals}{n+1}=j$ so that the new vertex in $\hier[G]{\cals}{n+1}$ is at location $j$.  Conditional on this and assuming $\hier[Y]{n,\cals}{i}\leq \hier[Y]{n}{i}$, the probability that each new edge in $\hier[G]{\cals}{n+1}$ connects to a vertex at location $i$, $1\leq i\leq N_{\cals}$, is
$$\frac{\hier[Y]{n,\cals}{i} a_{i,j}\gamma_{\cals}}{\sum_{k=1}^{0} \hier[Y]{n,\cals}{k}a_{k,j}} \leq  \frac{\hier[Y]{n,\cals}{i} b_{i,j}}{\sum_{k=0}^{N_{\cals}} \hier[Y]{n,\cals}{k}a_{k,j}}.$$

Now the numerator $$\hier[Y]{n,\cals}{i} b_{i,j}\leq \hier[Y]{n}{i} b_{i,j}=b_{i,j}\sum_{v:X_v\in S_i}\deg_{G_n}(v)\leq \sum_{v\in V(G_n):X_v\in S_i}\deg_{G_n}(v)\alpha(X_v,X_i).$$  For the denominator, define, for $1\leq i\leq N_{\cals}$, $\hier[Z]{n,\cals}{i}=\hier[Y]{n}{i}-\hier[Y]{n,\cals}{i}$, which is non-negative by our assumption.  Then the total degree of $G_n$ and $\hier[G]{\cals}{n}$ is the same, so $\hier[Y]{n,\cals}{0}=\sum_{k=1}^{N_{\cals}}\hier[Z]{n,\cals}{i}$, and thus we can write \begin{eqnarray*}\sum_{k=0}^{N_{\cals}} \hier[Y]{n,\cals}{k}a_{k,j} &=& \hier[Y]{n,\cals}{0}h+\sum_{k=1}^{N_{\cals}} \hier[Y]{n,\cals}{k}a_{k,j}\\ &=& \sum_{k=1}^{N_{\cals}} \left(\hier[Z]{n,\cals}{k}h+\hier[Y]{n,\cals}{k}a_{k,j}\right) \\ & \geq & \sum_{k=1}^{N_{\cals}}\hier[Y]{n}{k}a_{k,j}\\ &\geq & \sum_{v\in V(G_n)}\deg_{G_n}(v)\alpha(X_v,X_{n+1}).\end{eqnarray*}

Hence $$\frac{\hier[Y]{n,\cals}{i} a_{i,j}\gamma_{\cals}}{\sum_{k=1}^{0} \hier[Y]{n,\cals}{k}a_{k,j}}\leq \frac{\sum_{v\in V(G_n):X_v\in S_i}\deg_{G_n}(v)\alpha(X_v,X_{n+1})}{\sum_{v\in V(G_n)}\deg_{G_n}(v)\alpha(X_v,X_{n+1})}= \delta_n(S_i),$$ which is the probability that each new edge in $G_{n+1}$ connects to a vertex in $S_i$.  Hence, for $1\leq i\leq N_{\cals}$, the increase in the total degree at $i$ from $\hier[G]{\cals}n$ to $\hier[G]{\cals}{n+1}$ is at most the increase in the total degree in $S_i$ from $G_n$ to $G_{n+1}$, so $\hier[Y]{n+1,\cals}{i}\leq \hier[Y]{n+1}{i}$.

Hence the coupling ensures that $\hier[Y]{n,\cals}{i}\leq \hier[Y]{n}{i}$ for all $1\leq i\leq N_{\cals}$ and all $n$.

\subsection{Analysis of the coupled process}\label{coupled}

For $1\leq i\leq N_{\cals}$ $$\E(\hier[y]{n+1,\cals}{i}|\salj_n)-\hier[y]{n,\cals}{i} = \frac{2}{2(m+1+e_0/m)}\left(\half\mu_i+\half \sum_{j=1}^{N_{\cals}} \mu_j \frac{\hier[y]{n,\cals}{i}a_{i,j}\gamma_{\cals}}{\sum_{k=0}^{N_{\cals}} \hier[y]{n,\cals}{k}a_{k,j}}-\hier[y]{n,\cals}{i}\right) $$ and $$\E(\hier[y]{n+1,\cals}{0}|\salj_n)-\hier[y]{n,\cals}{N_{\cals}} = \frac{2}{2(m+1+e_0/m)}\left(\half \sum_{j=1}^{N_{\cals}} \mu_j \left( \frac{\hier[y]{n,\cals}{N_{\cals}}h}{\sum_{k=0}^{N_{\cals}} \hier[y]{n,\cals}{k}a_{k,j}}+(1-\gamma_{\cals})\sum_{\ell=1}^{N_{\cals}}\frac{\hier[y]{n,\cals}{\ell}a_{\ell,j}}{\sum_{k=0}^{N_{\cals}} \hier[y]{n,\cals}{k}a_{k,j}}\right)-\hier[y]{n,\cals}{i}\right),$$ giving, for $1\leq i\leq N_{\cals}$, $$\E(\hier[y]{n+1,\cals}{i}|\salj_n)-\hier[y]{n,\cals}{i}=\frac{2}{2(n+1+e_0/m)}\left(\hier[g]{\cals}{i}(\hier[y]{n,\cals}{})-\hier[y]{n,\cals}{i}\right),$$ where $\hier[g]{\cals}{}$ is a map from $\mathcal{Y}_{\cals}:=\{y:y\in\R^{N_{\cals}}, y_i\geq 0 \forall i, \sum_{k=1}^{N_{\cals}} y_k\leq 1\}$ to itself given by the $i$ co-ordinate being $$\hier[g]{\cals}{i}(y)=\half\mu_i+\half \sum_{j=1}^N \mu_j \frac{\gamma_{\cals} a_{i,j} y_i}{\sum_{k=1}^N y_k a_{k,j}}.$$ Alternatively $$ \E(\hier[y]{n+1,\cals}{i}|\salj_n)-\hier[y]{n,\cals}{i} = \frac{2}{2(n+1+e_0/m)}\hier[G]{\cals}{i}(\hier[y]{n,\cals}{i})$$ where $\hier[G]{\cals}{}(y)=\hier[g]{\cals}{}(y)-y$ and so its components are given by $$\hier[G]{\cals}{i}(y)=\half\mu_i+\half \sum_{j=1}^N \mu_j \frac{\gamma_{\cals}a_{i,j} y_i}{\sum_{k=1}^{N_{\cals}} y_k a_{k,j}}-y_i.$$  (Note that $\hier[y]{n,\cals}{0}=1-\sum_{k=1}^{N_{\cals}}\hier[y]{n,\cals}k$.)

\begin{prop}\label{lyapunovdustbin}For each $i$, $0\leq i\leq N_{\cals}$, there exists $\hier[\phi]{\cals}{i}\in\R^+$ such that as $n\to\infty$ we have $$\frac{2(mn+e_0)}{\deg_{\hier[G]{\cals}n}(v)}\pr(\hier[V]{n+1,\cals}{\iota}=v)=\hier[\phi]{\cals}{i}+o(1),$$ almost surely.
\end{prop}

\begin{proof}
For $y \in \mathcal{Y}_{\cals}$, let $$V(y)=\sum_{k=1}^{N_{\cals}}y_k-\half\sum_{j=1}^{N_{\cals}}\mu_j\left(\log y_j+\log\sum_{k=1}^{N_{\cals}} y_k a_{k,j}\right).$$
Then for $1\leq i\leq N_{\cals}$, $$G_i(y)=-y_i\frac{\partial}{\partial y_i}V(y),$$ and as $y_i>0$ this means that $V$ is a Lyapunov function for $G$. Again $V$ is a convex function and it tends to infinity as $y_i\to 0$, so it has a unique minimum, at a point which we will call $\hier[\nu]{\cals}{} \in \mathcal{Y}_{\cals}$.

Proposition 2.18 of \cite{pemantlesurvey} now gives $\hier[y]{n,\cals}{}\to \hier[\nu]{\cals}{}$ a.s. as $n\to\infty$.

This shows that for vertex $v$ at location $i$, $0\leq i\leq N_{\cals}$, $$\frac{2(mn+e_0)}{\deg_{\hier[G]{\cals}n}(v)}\pr(\hier[V]{n+1,\cals}{\iota}=v|X_{n+1}=j)=\frac{a_{i,j}\gamma_{\cals}}{\sum_{k=1}^{0} a_{k,j}\hier[\nu]{\cals}{k}}+o(1),$$ giving $$\frac{2(mn+e_0)}{\deg_{\hier[G]{\cals}n}(v)}\pr(\hier[V]{n+1,\cals}{\iota}=v)=\sum_{j=1}^N \mu_j \frac{a_{i,j}\gamma_{\cals}}{\sum_{k=0}^{N_{\cals}} a_{k,j}\hier[\nu]{\cals}{k}}+o(1)=\hier[\phi]{\cals}{i}+o(1),$$ where we define $$\hier[\phi]{\cals}{i} := \sum_{j=1}^N \mu_j \frac{a_{i,j}\gamma_{\cals}}{\sum_{k=1}^{0} a_{k,j}\hier[\nu]{\cals}{k}}.$$
\end{proof}

Let $$t=\frac{\inf_{1\leq i,j \leq N_{\cals}}b_{i,j}}{h}=\frac{\inf_{x,y\in S}\alpha(x,y)}{\sup_{x,y\in S}\alpha(x,y)}.$$  The conditions of Theorem \ref{overall} ensure that $t>0$.

\begin{prop}\label{nuphibound}We have $\hier[\phi]{\cals}{0}\leq \frac{2}{1+t}$ and $\hier[\nu]{\cals}{0}\leq \frac{(1-\gamma_{\cals})(1+t)}{2t}$.
\end{prop}

\begin{proof}
We have $$\hier[\nu]{\cals}{0}=\half((1-\gamma_{\cals})+\hier[\nu]{\cals}{0}\hier[\phi]{\cals}{0}),$$ giving \begin{equation}\label{nuphi}\hier[\nu]{\cals}{0}=\frac{1-\gamma_{\cals}}{2-\hier[\phi]{\cals}{0}}.\end{equation}

We also have $$\sum_{i=0}^{N_{\cals}} \hier[\nu]{\cals}{i} \hier[\phi]{\cals}{i}=1,$$ and by the definition of $t$ we have $$\hier[\phi]{\cals}{i}\geq t\hier[\phi]{\cals}{0}.$$  By $\sum_{i=0}^{N_{\cals}}\hier[\nu]{\cals}{i}=1$ we obtain \begin{equation}\label{phi0}\hier[\phi]{\cals}{0}(\hier[\nu]{\cals}{0}+t(1-\hier[\nu]{\cals}{0}))\leq 1.\end{equation}
Now, each new edge has at least one endpoint not at a vertex at location $0$, so $\hier[\nu]{\cals}{0}\leq \half$.  Hence \eqref{phi0} implies $$\hier[\phi]{\cals}{0}\leq \frac{2}{1+t},$$ and hence by \eqref{nuphi} $$\hier[\nu]{\cals}{0}\leq \frac{1-\gamma_{\cals}}{2-\frac{2}{1+t}}=\frac{(1-\gamma_{\cals})(1+t)}{2t}.$$
\end{proof}

\subsection{Approximating $S$}

We use the coupling in the previous section to complete the proof of Theorem \ref{overall}.

\begin{prop}\label{nuborel}Let $A\subseteq S$ be a Borel set.  Then there exists $\hat{\nu}(A)$ such that as $n\to\infty$ $\delta_n(A) \to \hat{\nu}(A)$, almost surely.\end{prop}

\begin{proof}Given $\epsilon>0$ we can construct a partition $\cals=\{S_1,S_2,\ldots,S_{N_{\cals}}\}$ of $S$ where each set $S_i,1\leq i\leq N_{\cals}$ has diameter at most $\epsilon$ and such that $A$ is the union of sets in $\cals$.  Then assuming $\log \alpha$ is a Lipschitz function (in both components) with Lipschitz constant $K$, we have for $1\leq i,j\leq N_{\cals}$ that $0\leq \log a_{i,j}-\log b_{i,j}\leq 2K\epsilon$ and so $\gamma_{\cals}\geq e^{-2K\epsilon}$.  The analysis in section \ref{coupled} shows that $$\delta_n(A)=\frac{1}{2(mn+e_0)}\sum_{i:S_i\subseteq A} \hier[Y]{n}{i}\geq \frac{1}{2(mn+e_0)}\sum_{i:S_i\subseteq A} \hier[Y]{n,\cals}{i}$$ and similarly that $$1-\delta_n(A)\geq \frac{1}{2(mn+e_0)}\sum_{i:S_i\subseteq A^c}\hier[Y]{n,\cals}{i}.$$  Furthermore as $n\to\infty$ $$\frac{1}{2(mn+e_0)}\sum_{i:S_i\subseteq A} \hier[Y]{n,\cals}{i}\to \sum_{i:S_i\subseteq A} \hier[\nu]{\cals}{i}$$ and $$\frac{1}{2(mn+e_0)}\sum_{i:S_i\subseteq A^c} \hier[Y]{n,\cals}{i} \to  \sum_{i:S_i\subseteq A^c} \hier[\nu]{\cals}{i}.$$  But by Proposition \ref{nuphibound}, \begin{eqnarray*} \sum_{i:S_i\subseteq A^c} \hier[\nu]{\cals}{i} &\geq & 1-\frac{(1-\gamma_{\cals})(1+t)}{2t}-\sum_{i:S_i\subseteq A} \hier[\nu]{\cals}{i} \\ &\geq& 1-\frac{(1-e^{2K\epsilon})(1+t)}{2t}-\sum_{i:S_i\subseteq A} \hier[\nu]{\cals}{i}.\end{eqnarray*}

So, almost surely, $$\liminf_{n\to\infty} \delta_n(A) \geq \sum_{i:S_i\subseteq A} \hier[\nu]{\cals}{i}$$ and $$\limsup_{n\to\infty} \delta_n(A) \leq \sum_{i:S_i\subseteq A} \hier[\nu]{\cals}{i}+\frac{(1-e^{2K\epsilon})(1+t)}{2t}.$$  As $t$ and $K$ are constants, taking a sequence of partitions $\cals$ such that $\epsilon\to 0$ gives us the result, as there will be a subsequence such that $\sum_{i:S_i\subseteq A} \hier[\nu]{\cals}{i}$ is convergent.\end{proof}

\textit{Proof of Theorem \ref{overall}.} Proposition \ref{nuborel} implies, by applying it individually to
each element in the set of closed balls with rational radii at
points in a countable dense subset of $S$ (which exists because a
compact metric space is separable) and using these to approximate
closed subsets of $S$, that, $\pr$-almost surely, that we have
$\limsup_{n\to\infty}\delta_{n}(A)= \hat{\nu}(A)$ for
all closed $A\subseteq S$.

Now, for closed subsets $A$ of $S$, define $$\nu'(A)=\inf_{B\mbox{ open,}A\subseteq B}\hat{\nu}(B),$$ and for open subsets $A$ of $S$ define $\nu'(A)=1-\nu'(A^c)$.  Then for closed sets $A$ we have $\nu'(A)\geq \hat{\nu}(A)$, and for open sets $A$ we have $\nu'(A) \leq \hat{\nu}(A)$.  By compactness, there will be a subsequence of $(\delta_n)_{n\in\N}$ which has a weak limit $\nu$ which is a probability measure on $S$.  Now  $\nu'(A)\geq \nu(A)$ for all closed sets $A$, but if $B$ is open with $A\subseteq B$ then $\nu(B)\geq \nu'(B)$, giving $\nu'(A)\geq \nu(A)$ on taking infima, hence $\nu(A)=\nu'(A)$ for all open and closed subsets of $S$.  Hence, $\pr$-almost surely, $\delta_n$ converges weakly to $\nu$.

Finally, if we define $$\phi(u)=\int_S\frac{ \alpha(u,y)}{\int_S \alpha(x,y)\;d\nu(x)}\;d\mu(y)$$ then $$\pr(\hier[V]{n+1}{\iota}=v)\frac{2(mn+e_0)}{\deg_{G_n}(v)}=\phi(X_v)+o(1).$$ \qed

\subsection{Proof of Theorem \ref{infinitedd}}

Let $A\in\mathcal{B}(S)$ be a Borel set with $\mu(A)>0$.  Let $\hier[p]{n}{d,A}$ be the proportion of vertices in $G_n$ which are of degree $d$ and have locations in $A$.  Let $\phi_A=\sup_{x\in A}\phi(x)$ and $\psi_A=\inf_{x\in A}\phi(x)$.

Fix $\phi$, and assume $\mu(A)<\frac{2}{\phi}-1$.  [If $\mu(A)$ is larger than this, partition $A$ into smaller sets for which the condition does hold.]  Then consider a graph process $(\tilde{G}_n)_{n\in \N}$ where if vertex $v$ is located in $A$ then $\pr(\hier[\tilde{V}]{n+1}{\iota}=v)=\phi\frac{\deg_{G_n}(v)}{2(mn+e_0)}$ for $n$ large enough.  (For any $\epsilon>0$, the total degree in $A$ will be at most $2(mn+e_0)(\frac{\phi}{2}(1+\mu(A))+\epsilon)$ for $n$ large enough, so the condition on $\mu(A)$ ensures that this is possible.)  We do not assume independence of $\hier[\tilde{V}]{n+1}{\iota}$ for different $\iota$ but do assume that $\pr(\hier[\tilde{V}]{n+1}{\iota_1}=\hier[\tilde{V}]{n+1}{\iota_2})=O(n^{-1}).$  Letting $\hier[\tilde{p}]{n}{d,A}$ being the proportion of vertices of $\tilde{G}_n$ which are of degree $d$ and in $A$, then the same argument as in the proof of Theorem \ref{finitedd} in section \ref{finite} shows that $$\hier[\tilde{p}]{n}{d,A}\to \frac{2\mu(A)}{\phi}\frac{\Gamma(m+2\phi^{-1})\Gamma(d)}{
\Gamma(m)\Gamma(d+2\phi^{-1}+1)},$$ almost surely.

If $\phi_A<\phi$, then we can couple the geometric preferential attachment process $(G_n)_{n\in\N}$ to a process of the above form such that for vertices in $A$ the degree is always at least as high in $\tilde{G}_n$ as in $G_n$.  (Give the new vertex the same location in each process, and then it is always possible to ensure $\pr(\hier[\tilde{V}]{n+1}{\iota}=v)>\pr(\hier[V]{n+1}{\iota}=v)$ for $v\in A$ with $(\tilde{G}_n)_{n\in \N}$ as described above.)

This ensures that the proportion of vertices of $G_n$ which are in $A$ and of degree at most $d$ satisfies satisfies $$\sum_{k=m}^{d} \hier[p]{n}{k,A}\geq \sum_{k=m}^{d} \hier[\tilde{p}]{n}{k,A}.$$  Hence, almost surely, $$\liminf_{n\to\infty} \sum_{k=m}^{d} \hier[p]{n}{k,A}\geq \frac{2\mu(A)}{\phi}\frac{\Gamma(m+2\phi^{-1})}{\Gamma(m)}\sum_{k=m}^{d} \frac{\Gamma(k)}{
\Gamma(k+2\phi^{-1}+1)},$$ for any $\phi>\phi_A$.

Similarly, almost surely, $$\limsup_{n\to\infty} \sum_{k=m}^{d} \hier[p]{n}{k,A}\leq \frac{2\mu(A)}{\phi}\frac{\Gamma(m+2\phi^{-1})}{\Gamma(m)}\sum_{k=m}^{d} \frac{\Gamma(k)}{
\Gamma(k+2\phi^{-1}+1)},$$ for any $\phi<\psi_A$.\qed


\begin{thebibliography}{10}

\bibitem{scalefreedefs}
R.~Albert, A.-L. Barab{\'a}si, and H.~Jeong.
\newblock Mean-field theory for scale-free random networks.
\newblock {\em Physica {A}}, 272:173--187, 1999.

\bibitem{bbbcr}
N.~Berger, B.~Bollob\'as, C.~Borgs, J.~Chayes, , and O.~Riordan.
\newblock Degree distribution of the {FKP} network model.
\newblock In {\em Automata, {Languages} and {Programming}}.

\bibitem{sfdiameter}
B.~Bollob{\'a}s and O.~Riordan.
\newblock The diameter of a scale-free random graph.
\newblock {\em Combinatorica}, 24(1):5--34, 2004.

\bibitem{brst}
B.~Bollob{\'a}s, O.~Riordan, J.~Spencer, and G.~Tusn{\'a}dy.
\newblock The degree sequence of a scale-free random graph process.
\newblock {\em Random Structures and Algorithms}, 18:279--290, 2001.

\bibitem{fitness}
C.~Borgs, J.~Chayes, C.~Daskalakis, and S.~Roch.
\newblock First to market is not everything: an analysis of preferential
  attachment with fitness.
\newblock In {\em S{TOC}'07---{P}roceedings of the 39th {A}nnual {ACM}
  {S}ymposium on {T}heory of {C}omputing}, pages 135--144. ACM, New York, 2007.

\bibitem{FFV1}
A.~D. Flaxman, A.~M. Frieze, and J.~Vera.
\newblock A geometric preferential attachment model of networks.
\newblock {\em Internet Math.}, 3:187--205, 2006.

\bibitem{FFV2}
A.~D. Flaxman, A.~M. Frieze, and J.~Vera.
\newblock A geometric preferential attachment model of networks {II}.
\newblock {\em Internet Math.}, 4:87--112, 2007.

\bibitem{geopref4}
J.~Jordan.
\newblock Degree sequences of geometric preferential attachment graphs.
\newblock {\em Adv. Appl. Prob.}, 42:319--330, 2010.

\bibitem{pemantlesurvey}
R.~Pemantle.
\newblock A survey of random processes with reinforcement.
\newblock {\em Probability Surveys}, 4:1--79, 2007.

\bibitem{penrosergg}
M.~Penrose.
\newblock {\em Random Geometric Graphs}.
\newblock Oxford University Press, 2003.

\bibitem{wadeong}
A.~R. Wade.
\newblock Asymptotic theory for the multidimensional random on-line
  nearest-neighbour graph.
\newblock {\em Stochastic Process. Appl.}, 119(6):1889--1911, 2009.

\end{thebibliography}
\end{document}